\documentclass[a4paper]{amsart}
\usepackage{graphicx,amsfonts,amssymb,amsmath}
\usepackage[all]{xypic}
\begin{document}


\newcommand{\dynkin}[5]{
\xymatrix@C=10pt{
{}\POS[]*++!D{\scriptstyle #1}*{\circ} \ar@{-}[r] &
{}\POS[]*++!D{\scriptstyle #2}*{\times} \ar@{-}[r] &
{}\POS[]*++!D{\scriptstyle 0}*{\circ} \ar@{}[rr] \ar@{-}[r] &
*{\scriptscriptstyle\ \cdots\ } \ar@{-}[r] &
{}\POS[]*++!D{\scriptstyle #3}*{\circ}  \ar@{-}[r] &
{}\POS[]*++!D{\scriptstyle #4}*{\circ}  \ar@{-}[r] &
{}\POS[]*++!D{\scriptstyle #5}*{\circ}   \ar@{}[rr] \ar@{-}[r] &
*{\scriptscriptstyle\ \cdots\ } \ar@{-}[r] &
{}\POS[]*++!D{\scriptstyle 0}*{\circ}}
}



\newcommand{\smb}          
{\mbox{$
\begin{picture}(2,2)(0,0)
\put(0,2.5){\line(1,0){2.5}}
\put(0,0){\line(1,0){2.5}}
\put(0,0){\line(0,1){2.5}}
\put(2.5,0){\line(0,1){2.5}}
\end{picture}$}}

\newcommand{\syiiII}{\mbox{%
$\begin{picture}(7,6)(-1,0)
\put(0,2.5){\smb}
\put(2.5,2.5){\smb}
\put(0,0){\smb}
\put(2.5,0){\smb}
\end{picture}$}}


\newcommand{\mb}                         
{\mbox{$
\begin{picture}(4,4)(0,0)
\put(0,5){\line(1,0){5}}
\put(0,0){\line(1,0){5}}
\put(0,0){\line(0,1){5}}
\put(5,0){\line(0,1){5}}
\end{picture}$}}

\newcommand{\yiII}{\mbox{$
\begin{picture}(7,10)(-1,1)
\put(0,5){\mb}
\put(0,0){\mb}
\end{picture}
$}}

\newcommand{\yiIIi}{\mbox{$
\begin{picture}(12,10)(-1,1)
\put(0,10){\line(1,0){10}}
\put(0,5){\line(1,0){10}}
\put(0,0){\line(1,0){5}}
\put(0,0){\line(0,1){10}}
\put(5,0){\line(0,1){10}}
\put(10,5){\line(0,1){5}}
\end{picture}
$}}

\newcommand{\yiiII}{\mbox{$
\begin{picture}(13,10)(-1,1)
\put(0,10){\line(1,0){10}}
\put(0,5){\line(1,0){10}}
\put(0,0){\line(1,0){10}}
\put(0,0){\line(0,1){10}}
\put(10,0){\line(0,1){10}}
\put(5,0){\line(0,1){10}}
\end{picture}
$}}

\newcommand{\yiIIIi}{\mbox{$
\begin{picture}(13,15)(-1,1)
\put(0,15){\line(1,0){10}}
\put(0,10){\line(1,0){10}}
\put(0,5){\line(1,0){5}}
\put(0,0){\line(1,0){5}}
\put(0,0){\line(0,1){15}}
\put(10,10){\line(0,1){5}}
\put(5,0){\line(0,1){15}}
\end{picture}
$}}


\newcommand{\yii}{ \mbox{$
\mbox{$\begin{picture}(13,17)(-1,-1)
\put(0,15){\line(1,0){10}}
\put(0,10){\line(1,0){10}}
\put(0,0){\line(1,0){10}}
\put(0,-5){\line(1,0){10}}
%
\put(0,-5){\line(0,1){20}}
\put(10,-5){\line(0,1){20}}
\put(5,-5){\line(0,1){7}}
\put(5,8){\line(0,1){7}}
\put(1,2){.}
\put(1,4){.}
\put(1,6){.}
\put(6,2){.}
\put(6,4){.}
\put(6,6){.}
\end{picture}$}
$}}

\newcommand{\yiik}[1]{ \mbox{$\mbox{$\begin{picture}(11,13)(-1,-1)
\put(1,3){\mbox{\tiny ${#1}$}}
\put(4,9){\vector(0,1){6}}
 \put(4,1){\vector(0,-1){6}}
\end{picture}$}
\mbox{$\begin{picture}(13,13)(-1,-1)
\put(0,15){\line(1,0){10}}
\put(0,10){\line(1,0){10}}
\put(0,0){\line(1,0){10}}
\put(0,-5){\line(1,0){10}}
%
\put(0,-5){\line(0,1){20}}
\put(10,-5){\line(0,1){20}}
\put(5,-5){\line(0,1){7}}
\put(5,8){\line(0,1){7}}
\put(1,2){.}
\put(1,4){.}
\put(1,6){.}
\put(6,2){.}
\put(6,4){.}
\put(6,6){.}
\end{picture}$}
$}}

\newcommand{\yiikwide}[1]{ \mbox{$\mbox{$\begin{picture}(21,13)(-1,-1)
\put(1,3){\mbox{\tiny ${#1}$}}
\put(10,9){\vector(0,1){6}}
 \put(10,1){\vector(0,-1){6}}
\end{picture}$}
\mbox{$\begin{picture}(13,13)(-1,-1)
\put(0,15){\line(1,0){10}}
\put(0,10){\line(1,0){10}}
\put(0,0){\line(1,0){10}}
\put(0,-5){\line(1,0){10}}
%
\put(0,-5){\line(0,1){20}}
\put(10,-5){\line(0,1){20}}
\put(5,-5){\line(0,1){7}}
\put(5,8){\line(0,1){7}}
\put(1,2){.}
\put(1,4){.}
\put(1,6){.}
\put(6,2){.}
\put(6,4){.}
\put(6,6){.}
\end{picture}$}
$}}

\newcommand{\yiia}{ \mbox{$
\mbox{$\begin{picture}(13,13)(-1,-3.5)
\put(0,15){\line(1,0){10}}
\put(0,10){\line(1,0){10}}
\put(0,0){\line(1,0){10}}
\put(0,-5){\line(1,0){10}}
\put(0,-10){\line(1,0){5}}
%
\put(0,-10){\line(0,1){25}}
\put(10,-5){\line(0,1){20}}
\put(5,-10){\line(0,1){12}}
\put(5,8){\line(0,1){7}}
\put(1,2){.}
\put(1,4){.}
\put(1,6){.}
\put(6,2){.}
\put(6,4){.}
\put(6,6){.}
\end{picture}$}
$}}

\newcommand{\yiiaa}{ \mbox{$
\mbox{$\begin{picture}(13,19)(-1,-6)
\put(0,15){\line(1,0){10}}
\put(0,10){\line(1,0){10}}
\put(0,0){\line(1,0){10}}
\put(0,-5){\line(1,0){10}}
\put(0,-10){\line(1,0){5}}
\put(0,-15){\line(1,0){5}}
%
\put(0,-15){\line(0,1){30}}
\put(10,-5){\line(0,1){20}}
\put(5,-15){\line(0,1){17}}
\put(5,8){\line(0,1){7}}
\put(1,2){.}
\put(1,4){.}
\put(1,6){.}
\put(6,2){.}
\put(6,4){.}
\put(6,6){.}
\end{picture}$}
$}}

\newcommand{\yiiaka}[1]{ \mbox{$\mbox{$\begin{picture}(20,21)(-1,-7)
\put(0,0){\tiny ${#1}$}
\put(9,10){\vector(0,1){6}}
 \put(9,-8){\vector(0,-1){6}}
\end{picture}$}
\mbox{$\begin{picture}(13,21)(-1,-7)
\put(0,15){\line(1,0){10}}
\put(0,10){\line(1,0){10}}
\put(5,0){\line(1,0){5}}
\put(5,-5){\line(1,0){5}}
%
\put(0,-15){\line(0,1){30}}
\put(10,-5){\line(0,1){20}}
\put(5,-15){\line(0,1){17}}
\put(5,8){\line(0,1){7}}
\put(0,-15){\line(1,0){5}}
\put(0,-10){\line(1,0){5}}
\put(1,-4){.}
\put(1,-6){.}
\put(1,-8){.}
\put(1,2){.}
\put(1,4){.}
\put(1,6){.}
\put(6,2){.}
\put(6,4){.}
\put(6,6){.}
\end{picture}$}
$}}


\renewcommand{\arraystretch}{1,5}

\newcommand{\ce}{{\mathcal E}}
\newcommand{\ph}{\phantom{'}}

\newcommand{\Do}{\mbox{\sf D}}
\newcommand{\M}{\mbox{\sf M}}
\newcommand{\Pe}{{\sf P}}

\newcommand{\ep}{{\epsilon}}
\newcommand{\dt}{{\delta}}
\newcommand{\al}{{\alpha}}
\newcommand{\bt}{{\beta}}
\newcommand{\fii}{{\varphi}}
\newcommand{\om}{\omega}
\newcommand{\Om}{\Omega}
\newcommand{\sig}{\sigma}

\newcommand{\asl}{\mathfrak{sl}}
\newcommand{\agl}{\mathfrak{gl}}
\newcommand{\g}{{\mathfrak g}}
\newcommand{\G}{{\mathcal G}}

\newcommand{\C}{{\mathbb C}}
\newcommand{\R}{{\mathbb R}}

\newcommand{\Ca}{{\mathcal{C}}}
\newcommand{\id}{{\textup{id}}}
\newcommand{\Alt}{{\operatorname{Alt}}}
\newcommand{\Ad}{{\operatorname{Ad}}}

\newcommand{\bal}{\boldsymbol{\al}}
\newcommand{\bbt}{\boldsymbol{\bt}}
\newcommand{\dbal}{{\bf \dot{\bal}}}
\newcommand{\ddbal}{{\bf \ddot{\bal}}}
\newcommand{\dbbt}{{\bf \dot{\bbt}}}
\newcommand{\ddbbt}{{\bf \ddot{\bbt}}}

\newcommand{\bA}{{\bf A}}
\newcommand{\bB}{{\bf B}}
\newcommand{\dbA}{{\bf \dot{A}}}
\newcommand{\ddbA}{{\bf \ddot{A}}}
\newcommand{\dbB}{{\bf \dot{B}}}
\newcommand{\ddbB}{{\bf \ddot{B}}}
\newcommand{\dddbB}{{\bf \dddot{B}}}

\newcommand{\XX}{{\mathbb X}}
\newcommand{\WW}{{\mathbb W}}
\newcommand{\YY}{{\mathbb Y}}

\newcommand{\lpl}
{\mbox{$
\begin{picture}(12.7,8)(-.5,-1)
\put(2,0.1){$+$}
\put(6.2,2.5){\oval(8,8)[l]}
\end{picture}$}}

\newcommand{\tpl}
{\mbox{$
\begin{picture}(12.7,8)(-.5,-1)
\put(2.2,0.2){$+$}
\put(6,2.8){\oval(8,8)[t]}
\end{picture}$}}

\newcommand{\modker}{{  \quad \operatorname{ mod } \operatorname{ Ker} (\Pi_{\syiiII})}}


\newcounter{counter}
\numberwithin{counter}{section}

 \newtheorem{lem}[counter]{Lemma}
 \newtheorem{thm}[counter]{Theorem}
 \newtheorem{cor}[counter]{Corollary} 
 \newtheorem{prop}[counter]{Proposition}
 \newtheorem*{prop*}{Proposition}
 \newtheorem{rem}[counter]{Remark}
 \newtheorem*{rem*}{Remark}
 \newtheorem{defin}[counter]{Definition}


\title[A nonstandard Grassmannian operator in the presence of a torsion]{A nonstandard operator for almost Grassmannian geometries with a torsion}
\author{ALE\v S N\'AVRAT}

\begin{abstract}
A nonstandard invariant fourth order operator acting on functions on a manifold equipped with an almost Grassmannian structure with an arbitrary trorsion is found by means of the curved translation principle. This operator can be viewed as a Grassmannian analogue of the Paneitz operator well known from conformal geometry. It is obtained by translating a Grassmannian analogue of the Laplace operator to a certain tractor bundle with a specific tractor connection, which is not normal in general.
\end{abstract}

\maketitle

\section{Introduction}
An almost Grassmannian structure on a smooth  manifold $M$ is given by a fixed identification of the tangent bundle $TM$ with the tensor product of two auxiliary vector bundles, together with the identification of their top degree exterior powers.
It turns out that this is a specific example of a so called $|1|$-graded parabolic geometry, i.e. a Cartan geometry of type $(G,P)$, where $P\subset G$ is a parabolic subgroup of a semi-simple Lie group $G$ with its Lie algebra $\g$ endowed with a $|1|$-grading, and so the set of tools of parabolic geometries summarized in \cite{parabook} applies.  Such structures are also known under the name almost Hermitian symmetric structures. The prototypical example from this class is the conformal structure - probably at most studied structure since it is the natural setting for the physics of  massless particles and other  theories in Physics. It was also an original  inspiration and motivation for this paper. 

In this text, we will deal with almost Grassmannian structures for which one of the defining vector bundles has rank two and the other has rank $n\geq 2$. Such a structure is called to be of type $(2,n)$. It  is well-known that  on a four-dimensional manifold $M$, i.e. in the case $n=2$, it is even equivalent to the conformal structure on $M$. This equivalence can be seen from the description of the two structures as classical first-order G-structures. Namely, the identification of $TM$ with the tensor product of two bundles of rank two together with the identification of the top-degree forms yields the reduction of the frame bundle to the structure group $G_0=S(GL(2,\R)\times GL(2,\R))\subset GL(4,\R)$, and this group is known to be isomorphic to the conformal spin group $CSpin(2,2)$. 

Almost Grassmannian structures of type $(2,n)$ for $n>2$ are not equivalent to conformal structures. One of the main differences is that this structure generally does not allow the existence of a torsion-free connection. Nevertheless, one can still generalize some concepts from conformal geometry. One of them is that of a conformally invariant differential operator, like the conformal Laplacian (sometimes called the Yamabe
operator), the Maxwell operator, or the Paneitz operator. Such operators act on sections of natural vector bundles, and may be defined by universal natural
formulas which depend only on the conformal structure and not on any choice
of metric tensor from the conformal class. Similarly, invariant differential operators for almost Grassmannian structures are defined by universal natural formulas which depend only on the structure and not on any particular choice. It turns out that most of conformally invariant operators have their invariant analogues for our structures. 

The aim of this text is to construct an analogue of the Paneitz operator. It is an conformally invariant second power of the Laplace operator,  i.e. a differential operator
\[
\Delta^2: \ce\to\ce[-4],
\]
which acts on functions and has values in densities of conformal wieght -4. It is well known that Paneitz operator on a  locally flat manifold (the case of a manifold $M$ which is locally isomorphic to a homogeneous space $G/P$) corresponds to a nonstandard homomorphism of Verma modules. Moreover, for $n=2$ this is the so--called critical power of Laplace operator and it is of much more subtle nature. Namely, it is not strongly invariant, i.e. it cannot be constructed algebraically from so--called semi--holonomic Verma modules and its invariance can be seen as a consequence of Bianchi identities. For more details  see \cite{Boe} and \cite{Graham}.

Considering an almost Grassmannian geometry given by an identification $TM=E^*\otimes F$ where $E$ is a rank two vector bundle and $F$ is a rank $n\geq 2$ vector bundle, 
an analogue of the Paneitz operator is a fourth order invariant operator
$$
\square : \ce\to\yiiII F^*[-2],
$$
where $\yiiII F^*[-2]:=\yiiII F^*\otimes \ce[-2]$ is an irreducible subbundle of the bundle of differential four-forms which is induced by the representation $\yiiII\R^{n*}$ of $SL(n,\R)$, and the one--dimensional representation $\R[-2]:=(\Lambda^2\R^2)^2$ of $SL(2,\R)$. Indeed, for $n=2$ the target bundle is isomorphic to $(\Lambda^2F^*)^2[-2]=\ce[-4]$ and we have $\square=\Delta^2$. Similarly to Paneitz operator,  its Grassmannian analogue $\square$  also corresponds to a nonstandard homomorphism of Verma modules, see e.g. \cite{BoeCollingwood}. Moreover, it turns out that it is even not strongly invariant, \cite{Navrat}. This indicates that the construction of a curved analogue will be a difficult task if it exists.

The existence of a curved analogue has been proven in the torsion free case in \cite{Gover}. In this paper in Theorem \ref{thm} we generalize this result to the case of almost Grassmannian geometries with an arbitrary torsion. The construction is based on the curved translation principle, see \cite{translation}. We preceed along the lines of the construction of Paneitz operator given in \cite{Qcurv}.

The structure of the paper is as follows. First we recall basic facts about Grassmannian geometry in the realm of parabolic geometries and we explore the structure of some tractor bundles that we need in sequel. In the third section, we give the main result in Theorem \ref{thm} about the existence of an invariant operator and we describe its construction.

\section{Grassmannian tractor calculus}
In this section, we give a background on almost Grassmannian geometries that we need in sequel for the construction of a curved analogue of the nonstandard operator.

\subsection{Notation}
\label{Notation}
The Penrose's abstract index notation is used throughout the article. It allows explicit calculations without involving a choice of basis. The  two fundamental  bundles for Grassmannian structures are denoted by $\ce^{A}$ and $\ce^{A'}$, for their duals we write $\ce_{A}$ and $\ce_{A'}$. As usual,  $\ce^a$ and $\ce_a$ denote the tangent bundle and  the cotangent bundle  respectively, and  the tensor product of bundles is denoted by a concatanation of indices. In analogy with conformal geometry, the line bundle $\bigwedge\nolimits^{\!\!2}\ce^{A'}$ is denoted by $\ce[-1]$, and it is called the bundle of densities of weight $-1$. By dualizing and tensorizing, one defines the bundle of densities $\ce[w]$ for an arbitrary $w\in\mathbb{Z}.$

We also adopt the index notation for forms of \cite{prolongation} in order to obtain formulas in a closed form. Namely, the brackets in the usual notation $\ce_{[AB]}$ for second skew symmetric power of $\ce_A$ will be supressed by using two consecutive indices, i.e. it will be  denoted by  $\ce_{A^1A^2}$ which will be farther abbreviated to
$
\ce_\bA.
$
by using the multi-index $\bA=A^1A^2=[A^1A^2]$. Note that by use of consecutive indices we always mean the alternation over these indices.

\subsection{Almost Grassmannian geometry of type $(2,n)$}
In the abstract index notation,  the  identifications definining an almost Grassmannian structure of type $(2,n)$ read as
\[
\ce^a\cong
\ce_{A'}^A
\; \text{ and } \;
\bigwedge\nolimits^{\!\!2}\ce^{A'}\cong\bigwedge\nolimits^{\!\!n}\ce_A,
\]
 where we consider  $\ce^A$ has an arbitrary fibre dimension $n$ while 
\begin{equation}
\label{dim=2}
\dim\ce^{A'}=2.
\end{equation}
Equivalently, the structure can be defined as a classical first order G-structure with reduction of the structure group $GL(2n,\R)$ to its subgroup $S(GL(2,\R)\times GL(n,\R))$. 
It is well known that this structure corresponds to a $|1|$-graded normal parabolic geometry of type $(G,P)$, where $G=SL(2+n,\R)$ and $P$ is the stabilizer of $\R^2$ in $\R^{2+n}$. The  flat model $G/P$ of the geometry are Grassmannians $\operatorname{Gr}_2(\R^{2+n})$. In the diagram notation, the $|1|$-grading of Lie algebra $\g=\asl(2+n,\R)$ defining the geometry is given by the $A_{n+1}$-diagram with  the second node crossed. Hence $\g=\g_{-1}\oplus\g_0\oplus\g_1$, where
$$
\g_0\cong\mathfrak{s}(\agl (2,\R)\oplus\agl (n,\R)),\quad \g_1\cong(\g_{-1})^* \cong\R^2\otimes\R^{n*}
$$
and $\mathfrak{p}=\g_0\oplus\g_{1}$ is the subalgebra corresponding to parabolic subgroup $P$. 

There exist two basic invariants for the almost Grassmannian structure. In the language of parabolic geometries, these invariants are the two components of the harmonic curvature. In terms of representation theory, the first one lies in the intersection of kernels of the two possible contractions of
$
S^2\R^2\otimes\R^{2*}\otimes\Lambda^2\R^{n*}\otimes\R^n
$
while the second lies in  the kernel of a unique contraction on the space
$
\Lambda^2\R^2\otimes S^3\R^{n*}\otimes\R^n.
$
The former has homogeneity one and thus it corresponds to a torsion of a linear connection on the tangent bundle. The latter has homogeneity two and thus it corresponds to a component of its curvature. For details see \cite[section 4.1]{parabook}.

\subsection{Weyl connections}
\label{connections}
For each parabolic geometry there exist a class of distinguished connections that are invariant with respect to the structure, so called Weyl connections. They may be viewed as analogs of Levi-Civita connections in conformal geometry. We will use these connections to define differential operators. Precisely, we will use so called exact Weyl connections in order to get as simple formulas as possible, \cite{parabook}. These connections form an affine space modeled on exact one forms and they are characterized by the fact that they have the minimal torsion in the sense that the torsion is equal to the torsion invariant of the structure, i.e. the torsion of a Weyl connection satisfies
\begin{equation}
\label{torsion}
T_{ab}{}^c=T^{A'B'C}_{A\ph B\ph C'}\in \Gamma(\ce^{(A'B')C}_{[A\ph B]\ph C'}),
\;\; T^{A'R'C}_{A\ph B\ph R'}=T^{A'B'R}_{A\ph R\ph C'}=0.
\end{equation}
Our convention for the curvature $R_{ab}{}^c{}_d$ and the torsion $T_{ab}{}^c$ of a covariant derivative $\nabla_a$ on the tangent bundle is given by the equation
\[
2\nabla_{[a}\nabla_{b]}u^{c}=R_{ab}{}^c{}_d u^{d}+T_{ab}{}^d\nabla_d u^{c}.
\]
Obviously, a Weyl connection is  given by a product of connections on spinor bundles $\ce^A$ and $\ce_{A'}$. We fix the notation for their curvatures by equations
\[
\begin{aligned}[]
2\nabla_{[a}\nabla_{b]}u^{C'}&=R_{abD'}^{\phantom{ab}C'}u^{D'}+T_{ab}{}^d\nabla_d u^{C'},
\\
2\nabla_{[a}\nabla_{b]}u^{C}&=R_{abD}^{\phantom{ab}C}u^{D}+T_{ab}{}^d\nabla_d u^{C}.
\end{aligned}
\]
Hence 
$
R_{ab}{}^c{}_d=R_{abD}^{\phantom{ab}C}\dt^{D'}_{C'}-R_{abC'}^{\phantom{ab}D'}\dt^C_D.
$
Next, we define a Grassmannian  analogue of the conformal Weyl tensor as the following trace modification of the Riemannian curvature tensor 
\begin{equation}
\label{U_def}
\begin{aligned}
 U_{abD}^{\phantom{ab}C}&=R_{abD}^{\phantom{ab}C} 
+\Pe^{A'B'}_{AD}\dt^{C}_{B}-\Pe^{B'A'}_{BD}\dt^{C}_{A},
\\
U_{abD'}^{\phantom{ab}C'}&=R_{abD'}^{\phantom{ab}C'}
-\Pe^{A'C'}_{AB}\dt^{B'}_{D'}+\Pe^{B'C'}_{BA}\dt^{A'}_{D'},
\end{aligned}
\end{equation}
where  $\Pe_{ab}$ is the Rho tensor for the Weyl connection $\nabla_a$ \cite{parabook}. In the theory of parabolic geometries, $\Pe$ corresponds to the  $\g_1$  part of the Cartan connection and the curvature tensor $U$ is the $\g_0$ part of the curvature $\kappa$ of the Cartan connection. The normality of $\kappa$ results in a co-closedness condition for the analogue of the Weyl tensor  \cite{Gover}
\begin{equation}
\label{coclosedness}
U^{A'B'R}_{R\ph B\ph A}-U^{R'B'A'}_{A\ph B\ph R'}=0.
\end{equation}
According to the description mentioned above, the second invariant for the structure, the harmonic part of the curvature, is the component that fulfills
\begin{equation}
\label{harmonic_curv}
\rho^{\phantom{AB}D}_{ABC}\in\Gamma(\ce_{(ABC)}^{\phantom{(AB}D}[-1]), \; \rho^{\phantom{AB}R}_{ABR}=0.
\end{equation}

\subsection{Standard tractor and cotractor bundle}
\label{standard_tractors}
By definition, the standard tractor bundle $\ce^\al$ is induced by the standard representation of $\g$ on $\R^{n+2}$. The inclusion of the $\mathfrak{p}$--invariant subspace  $\R^2\hookrightarrow \R^{2+n}$ gives rise to the composition series 
\[
\xymatrix{\ce^{A'}\ar@{^{(}->}[r]^{X}&\ce^\al \ar@{->>}[r]^{Y}&\ce^A}
\]
which we write as $\ce^\al=\ce^A\lpl\ce^{A'}$.  Similarly, the composition series for the standard cotractor bundle reads as $\ce_\al=\ce_{A'}\lpl\ce_{A}$. Let us remark that a choice of the Weyl structure, or equivalently a choice of a nowhere vanishing section of $\ce[-1]$, splits the composition series of these bundles into a direct sum. Consequently, the composition series for all tractor bundles split and we can display their sections as `vectors'. In sequel, we focus on cotractors, i.e. sections of the cotractor bundles. A standard cotractor written  in a vector form reads $v_\al=(v_{A'}\: v_A)$.  Alternatively, the sections will be written in terms of injectors $Y^\al_A\in\ce^\al_A$, $X^\al_{A'}\in\ce^\al_{A'}$ and $Y^{A'}_\al\in\ce^{A'}_\al$, $X^A_\al\in\ce^A_\al$ respectively. We assume the injectors are normalized in such a way that $X^B_\al Y^\al_A=\dt{}_A^B$, $X^\al_{A'}Y^{B'}_\al =\dt{}_{A'}^{B'}$. Then  a standrad cotractor is given by
\[
v_\al
=Y^{A'}_\al v_{A'}+X^A_\al v_A.
\]
The restriction of the inducing representation  to $\g_1$ and $\g_{-1}$ gives rise to  actions of a $\fii\in T^*M$ and $\xi\in TM$  on sections of $\ce_\al.$ It is easy to see that the actions are given by
\[
\begin{aligned}
\Upsilon\bullet\begin{pmatrix}v_{A'} \\v_A\end{pmatrix}
&=\begin{pmatrix}0 \\-\Upsilon^{R'}_A v_{R'}\end{pmatrix}
&\text{and }&&
\xi\bullet\begin{pmatrix}v_{A'} \\v_A\end{pmatrix}
=\begin{pmatrix}-\xi^R_{A'}v_R \\0\end{pmatrix}.
\end{aligned}
\]
By \cite{parabook}, a general formula for the normal connection $\nabla^\mathcal{T}$ on a section $t=(t_i)$ of $\mathcal{T}$ reads as
$
(\nabla^\mathcal{T}_\xi t)_i=\nabla_{\xi}t_i+\Pe(\xi)\bullet t_{i-1}+\xi\bullet t_{i+1}.
$
Substituting the above equations we compute the normal connection on the standard cotractor bundle $\ce_\al$:
\[
\nabla^{A'}_A\begin{pmatrix}v_{C'} \\v_C\end{pmatrix}
=\begin{pmatrix}\nabla^{A'}_Av_{C'}-v_A\dt^{A'}_{C'} \\ \nabla^{A'}_Av_C+\Pe^{A'R'}_{AC}v_{R'}\end{pmatrix}.
\]
Note that $\nabla$ on the left hand side denotes the tractor connection while the same symbol on the right stands for the Weyl connection.
For the curvature of the tractor connection we use the notation 
$\Om_{ab}\sharp v_\gamma=(2\nabla_{[a}\nabla_{b]}-T_{ab}{}^e\nabla_e)v_\gamma$. Then one computes
\begin{equation}
\label{standard_curvature}
\Om_{ab}\sharp \begin{pmatrix}v_{C'} \\v_C\end{pmatrix}=
\begin{pmatrix}
T_{ab}{}^R_{C'}v_R-U_{ab}{}^{R'}_{C'}v_{R'} \\
-U_{ab}{}^R_Cv_R+Q_{ab}{}^{R'}_C v_{R'}
\end{pmatrix},
\end{equation}
where
$
Q_{abc}:=(d^\nabla\Pe)_{abc}=2\nabla_{[a}\Pe_{b]c}-T_{ab}{}^i\Pe_{ic}.
$
This tensor is usually called the (generalized) Cotton--York tensor to $\nabla$, see \cite{parabook}.

\subsection{Tractor bundle $\ce_{[\al\bt]}$}
\label{tractors}
The composition series for the second skewsymmetric power of the standard tractor bundle can be deduced from the composition series of $\ce_\al$. Using the isomorphism $\ce_{A'}\cong\ce^{A'}[1]$ it can be written as
\begin{equation}
\label{comp_series}
\ce_{[\al\bt]}=\ce[1]\lpl\ce^{A'}_A[1]\lpl\ce_{[AB]}.
\end{equation}
In terms of injectors, a section of $\ce_{[\al\bt]}$ reads as
\[
v_{\al\bt}=Y^{A'}_{[\al} Y^{B'}_{\bt]}\sig\ep_{A'B'}
+X^{A}_{[\al} Y^{B'}_{\bt]}\mu^{A'}_A\ep_{A'B'}
+X^{A}_{[\al} X^{B}_{\bt]}\rho_{AB}.
\]
An advantage of this notation is an easy computation of the actions of $\Upsilon\in T^*M$ and $\xi\in TM$. In the vector notation, they read as 
\[
\Upsilon\bullet\begin{pmatrix}\sig\\ \mu^{A'}_A \\ \rho_{AB}\end{pmatrix}
=\begin{pmatrix}0\\-2\Upsilon^{A'}_A \sig \\-\Upsilon_{R'[A}\mu^{R'}_{B]}\end{pmatrix}
\; \text{ and } \;
\xi\bullet\begin{pmatrix}\sig\\ \mu^{A'}_A \\ \rho_{AB}\end{pmatrix}
=\begin{pmatrix}-\frac12\xi^{R}_{R'} \mu^{R'}_R \\-2\xi^{RA'}\rho_{AR}\\0\end{pmatrix}.
\]
Substituting these equations into the general formula for the normal tractor connection, we get a formula for the connection on $\ce_{[\al\bt]}$:
\begin{equation}
\label{normal_connection}
\nabla^{A'}_A\begin{pmatrix}\sig\\ \mu^{C'}_C \\ \rho_{CD}\end{pmatrix}
=\begin{pmatrix}\nabla^{A'}_A \sig - \frac12\mu^{A'}_A\\
\nabla^{A'}_A\mu^{C'}_C  +2\Pe^{A'C'}_{AC} \sig +2 \ep^{A'C'}\rho_{AC}
\\ \nabla^{A'}_A\rho_{CD}-\Pe^{A'R'}_{A[C} \mu_{D]R'}\end{pmatrix}.
\end{equation}
\begin{rem}\normalfont \label{rem}
Let us remark that the invariance of tractor connections and their curvatures and all geometric objects we construct from them can be  also checked directly by computing their linearized transformation, denoted by $\dt$. 
By general formulas in propositions 5.1.5, 5.1.6 and 5.1.8 in \cite{parabook},   we get
\begin{align*}
\dt\nabla^{A'}_Av^{C'}&=-\Upsilon^{C'}_Av^{A'}, \;\;
\dt\nabla^{A'}_Av_{C'}=\Upsilon^{R'}_Av_{R'}\dt^{A'}_{C'},
\\
\dt\nabla^{A'}_Av_{C}&=-\Upsilon^{A'}_Cv_{A}, \;\;
\dt\nabla^{A'}_Av^{C}=\Upsilon^{A'}_Rv^{R}\dt^C_A.
\end{align*}
For $f\in\ce[w]$ we have $\dt(\nabla^{A'}_Af)=w\Upsilon^{A'}_Af$ and the transformation of derivatives of higher tensor products is obtained by the Leibnitz rule. The  curvature terms defined in \ref{connections} and \ref{tractors} transform as follows.
\begin{align*}
\dt U_{ab}{}^{C'}_{D'}&=\Upsilon^{C'}_R T_{ab}{}^{R}_{D'},
\\
\dt U_{ab}{}^{C}_{D}&=-\Upsilon^{R'}_D T_{ab}{}_{R'}^C,
\\
\dt Q_{abc}&=-\Upsilon^{R'}_C U_{ab}{}^{C'}_{R'}+\Upsilon^{C'}_R U_{ab}{}^R_C.
\end{align*}
Applying these basic formulas one can check that all invariant objects constructed in this paper indeed satisfy $\dt=0$.
\end{rem}

\section{The nonstandard operator}

In terms of our notation, a Grassmannian analogue of the Paneitz operator is an invariant operator
\[
\square_{\bA \bB}: \ce\to\yiiII\ce_{\bA \bB}[-2]
\]
such that its leading part is given by the projection of a nonzero multiple of $\Delta_{\bA}\Delta_{\bB}$ to the target bundle. 
We construct such an operator by translating a Grassmannian analogue of the conformally invariant version of the Laplace operator  (also called Yamabe operator). It  is an operator
$
\square_{AB}:\ce[-1]\to\ce_{[AB]}[-2]
$
which in terms of a Weyl connection reads
\begin{equation}
\label{Laplace}
\square_{AB}=\nabla_{R'[A}\nabla^{R'}_{B]}-\Pe_{R'[AB]}^{\phantom{R'[A}R'}.
\end{equation}
It is easy to see that  this formula is invariant in a strong sense and thus defines an invariant operator on each tractor bundle $\mathcal{T}[-1]$ of weight $-1$ by coupling $\nabla$ to a tractor connection on $\mathcal{T}$.  Assume $\mathcal{T}=\ce_{[\gamma\dt]}$.  Then by \eqref{comp_series} the composition series for the initial bundle of $\square_{AB}$ is 
$\ce_{[\gamma\dt]}[-1]=\ce\lpl\ce^{C'}_C\lpl\ce_{[CD]}[-1]$
while the target bundle decomposes as follows. Note that we use the notation for skew symmetric products introduced in \ref{Notation}.
\[
\ce_{\bA}\otimes\ce_{[\gamma\dt]}[-2]=
\ce_{\bA}[-1]\lpl
\begin{matrix}\ce^{A'}\otimes\yiIIi\ce_{\bA B}[-1]\\\oplus\\ \ce^{A'}\otimes\ce_{[\bA B]}[-1]\end{matrix}
\lpl
\begin{matrix}\yiiII\ce_{\bA \bB}[-2]\\\oplus\\ \yiIIIi\ce_{\bA \bB}[-2]\\ \oplus\\ \ce_{[\bA \bB]}[-2]\end{matrix}
\]
We will show that for a suitable connection on $\ce_{[\gamma\dt]}$ we obtain the Grassmannian nonstandard operator as a composition
\[
\xymatrix{
\ce \ar[r]^-{S_{\gamma\dt}} &\ce_{[\gamma\dt]}[-1] \ar[r]^-{\widetilde\square_{\bA}} & \ce_{\bA[\gamma\dt]}[-2] \ar@{->>}[r] & \yiiII\ce_{\bA \bB}[-2],
}
\]
where $S_{\gamma\dt}$ is a differential splitting operator and $\widetilde\square_{\bA}$ is the second order operator  defined by \eqref{Laplace} but for a certain tractor connection $\widetilde{\nabla}^{A'}_A$. It turns out that this connection is not normal in general. Namely, in terms of injectors $X^A_\al, Y^{A'}_\al$ defined in section \ref{standard_tractors} and in temrs of curvature tensors  $U$ and $Q$  defined in sections \ref{connections} and \ref{tractors}, it is given by
\begin{equation}
\label{connection}
\begin{aligned}
\widetilde{\nabla}^{A'}_A\left(Y_{C'[\al}Y^{C'}_{\bt]}\right)&=X^{C}_{[\al}Y_{\bt]D'}\left(-2\Pe^{A'D'}_{AC}
+4\Pe^{(A'D')}_{[AC]}+2U^{R'A'D'}_{CAR'}\right)
\\& \mkern-18mu\mkern-18mu +X^C_{[\al}X^{D}_{\bt]}\left(2Q^{A'R'}_{A[CD]R'}-4\nabla_{R'[C}U^{S'(A'R')}_{D]AS'}
+4\nabla_{R'[C}U^{S'[A'R']}_{D]AS'}\right),
\\
\widetilde{\nabla}^{A'}_A\left(X^{C}_{[\al}Y^{D'}_{\bt]}\right)&=Y_{C'[\al}Y^{C'}_{\bt]}\left(\frac12\ep^{A'D'}\dt^C_A\right)
+X^R_{[\al}Y^{R'}_{\bt]}\left(T^{A'D'C}_{ARR'}\right)
\\&\mkern-18mu\mkern-18mu +X^R_{[\al}X^S_{\bt]}\left(-\Pe^{A'D'}_{A[R}\dt^C_{S]}-U^{A'D'C}_{A[RS]}\right),
\\
\widetilde{\nabla}^{A'}_A\left(X^{C}_{[\al}X^{D}_{\bt]}\right)&=X^D_{[\al}Y^{A'}_{\bt]}\left(2\dt^C_A\right).
\end{aligned}
\end{equation}
We formulate the  precise result in the following theorem, where we use again the notation  from section \ref{Notation} for skew symmetric products.
\begin{thm}
\label{thm}
Let $M$ be an almost Grassmannian structure of type $(2,n)$ with an arbitrary torsion. There  exists a fourth order invariant operator,
\[
\square_{\bA \bB}: \ce\to\yiiII\ce_{\bA \bB}[-2]
\]
which is nontrivial on flat structures, i.e. there exists a
curved analogue of the nonstandard operator on functions.
\\
The operator is obtained by a translation of operator \eqref{Laplace} to the tractor bundle $\ce_{[\al\bt]}[-1]$ with  connection \eqref{connection}. In terms of a Weyl connection, it  is given by 
\begin{equation}
\label{formula}
\begin{aligned}
\widetilde\square_{\bA}S_{\bB}&=\frac12\Delta_{\bA}\Delta_{\bB}
+\nabla_{R'A^1}\Pe^{R'}_{A^2S'B^1}\nabla^{S'}_{B^2}
+\Pe^{R'S'}_{A^1B^1}\nabla_{A^2R'}\nabla_{B^2S'} 
\\
&
-\Pe_{R'A^1B^1}^{\phantom{R'}R'}\Delta_{A^2B^2}
+\frac12\Pe_{R'\bA}^{\phantom{R'}R'}\Delta_{\bB}
\\
&
+U^{R'S'R}_{A^1\bB}\nabla_{A^2R'}\nabla_{RS'}
+U^{\phantom{S'A}S'R}_{S'A^1\bB}\Delta_{A^2R}
\\
&
-2\nabla_{S'B^1}U^{P'(R'S')}_{B^2A^1P'}\nabla_{A^2R'}
+2\nabla_{S'B^1}U^{P'[R'S']}_{B^2A^1P'}\nabla_{A^2R'}
\\
&
-\nabla_{R'A^1}U^{R'S'R}_{A^2\bB}\nabla_{RS'}
+\Pe_{B'A^1B^1}^{\phantom{B'A}R'}T^{B'\phantom{DR}R}_{A^2B^2R'S'}\nabla^{S'}_R 
\\ &
+U_{B'A^1\bB}^{\phantom{B'A}R'R}T^{B'\phantom{R'R}S}_{A^2RR'S'}\nabla^{S'}_S 
+Q^{R'S'}_{A^1\bB S'}\nabla_{A^2R'},
\end{aligned}
\end{equation}
followed by a projection to the target bundle.
\end{thm}

\subsection{Proof of theorem \ref{thm} in the torsion free case}
The formula for  $\widetilde{\nabla}^{A'}_A$ simplifies considerably in the case that the torsion vanish. It follows from the general theory, see Theorem 4.1.1 in \cite{parabook}, that  the curvature $U_{abc}^{\phantom{ab}d}$ defined by \eqref{U_def} has only one nonzero irreducible component in such a case, namely the harmonic part $\rho^{\phantom{AB}D}_{ABC}$ defined by \eqref{harmonic_curv}. Moreover, the tensor $Q_{abc}$ can be expressed purely  in terms of this curvature and $\Pe_{[ab]}$  vanish. All these facts can be also checked directly using the consequences of the Bianchi identity  summarized in appendix, Lemma \ref{B}.  But the harmonic curvature $\rho^{\phantom{AB}D}_{ABC}$ is symmetric in three lower unprimed indices by definition and thus it cannot appear in the formula for  $\widetilde{\nabla}^{A'}_A$. 
Hence in the torsion free case  the formula \eqref{connection} simplifies to
\[
\begin{aligned}
\widetilde{\nabla}^{A'}_A\left(Y_{C'[\al}Y^{C'}_{\bt]}\right)&=X^{C}_{[\al}Y_{\bt]D'}\left(-2\Pe^{A'D'}_{AC}\right)
\\
\widetilde{\nabla}^{A'}_A\left(X^{C}_{[\al}Y^{D'}_{\bt]}\right)&
=Y_{C'[\al}Y^{C'}_{\bt]}\left(\frac12\ep^{A'D'}\dt^C_A\right)
+X^R_{[\al}X^S_{\bt]}\left(-\Pe^{A'D'}_{A[R}\dt^C_{S]}\right),
\\
\widetilde{\nabla}^{A'}_A\left(X^{C}_{[\al}X^{D}_{\bt]}\right)&=X^D_{[\al}Y^{A'}_{\bt]}\left(2\dt^C_A\right),
\end{aligned}
\]
which coincides with formula \eqref{normal_connection} for the normal tractor connection. Thus the following proposition proves the theorem in the torsion free case. 

\begin{prop}
In the torsion free case, the translation of operator \eqref{Laplace} to the tractor bundle $\ce_{[\al\bt]}[-1]$ with respect to the normal tractor connection on that bundle yields a curved analogue of the nonstandard operator on functions.
\end{prop}
\begin{proof}
We consider a differential splitting $S_{\al\bt}: \ce\to\ce_{[\al\bt]}[-1]$ defined by
\begin{equation}
\label{split}
S_{\al\bt}f=\begin{pmatrix}0\\ \nabla^{A'}_A f\\ \frac12\Delta_{AB}f\end{pmatrix},
\end{equation}
where we denote
$
\Delta_{AB}:=\nabla_{R'[A}\nabla^{R'}_{B]}.
$
It is an analogue of the conformal tractor--D operator, and it is easy to check its invariance by a direct computation.
Applying formula \eqref{normal_connection} for the normal tractor connection  twice we get
\[
\widetilde\nabla^{A'}_A\widetilde\nabla^{B'}_BS_{\gamma\dt}
=\begin{pmatrix}
-\frac12\nabla^{A'}_A\nabla^{B'}_B  -\frac12\nabla^{B'}_B\nabla^{A'}_A 
-\frac12\ep^{B'A'}\Delta_{BA}\\
\nabla^{A'}_A\nabla^{B'}_B\nabla^{C'}_C  + \ep^{B'C'}\nabla^{A'}_A\Delta_{BC}
-\Pe^{A'C'}_{AC}\nabla^{B'}_B  \\
 +\ep^{A'C'}\nabla^{B'}_B\Delta_{AC}   -2\ep^{A'C'} \Pe^{B'R'}_{B[A}\nabla_{C]R'}\\
 \frac12\nabla^{A'}_A\nabla^{B'}_B\Delta_{CD}
-\nabla^{A'}_A\Pe^{B'R'}_{B[C} \nabla_{D]R'}
\\
-\Pe^{A'R'}_{A[C}\nabla^{B'}_{|B|}\nabla_{D]R'}
-\Pe^{A'B'}_{A[C}\Delta_{D]B}
\end{pmatrix}.
\]
Now substituting the corresponding term in defining equation \eqref{Laplace} of $\square_{AB}$, we get zero in  top slot of $\widetilde\square_{AB}S_{\gamma\dt}f$, while in the middle slot we have
\begin{equation}
\label{obstruction}
\begin{aligned}
\widetilde\square_{AB}S^{C'}_{C}&=
\Delta_{AB}\nabla^{C'}_C  
+2\nabla^{C'}_{[A}\Delta_{B]C} 
 \\
&-\Pe^{C'R'}_{[AB]}\nabla_{CR'}+\Pe^{C'R'}_{[A|C|}\nabla_{B]R'}
+\Pe^{R'C'}_{[A|C|}\nabla_{B]R'}+\Pe^{R'}_{[AB]R'}\nabla^{C'}_{C}.
\end{aligned}
\end{equation}
We will show that its action on function also gives zero if the torsion vanishes. First observe that its third order symbol vanishes due to \eqref{dim=2}. Thus the principal symbol is of the first order and contains a curvature. And since it is invariant by construction, the terms containing Rho tensors cancel out and the principal symbol must be given purely in terms of  the harmonic curvature $\rho^{\phantom{AB}D}_{ABC}$. But this curvature component cannot appear because of its symmetries \eqref{harmonic_curv}. Hence we must have
\[
\widetilde\square_{AB}S^{C'}_{C} f=0.
\]
and so the formula in the bottom slot of tractor $\widetilde\square_{AB}S_{\gamma\dt}$ is  invariant.  Making this formula explicit and using then notation introduced in \ref{Notation} we get 
\[
\begin{matrix}
\widetilde\square_{\bA}S_{\bB}=\frac12\Delta_{\bA}\Delta_{\bB}
+\nabla_{R'A^1}\Pe^{R'}_{A^2S'B^1}\nabla^{S'}_{B^2}
+\Pe^{R'S'}_{A^1B^1}\nabla_{A^2R'}\nabla_{B^2S'} \\
-\Pe_{R'A^1B^1}^{\phantom{R'}R'}\Delta_{A^2B^2}
+\frac12\Pe_{R'\bA}^{\phantom{R'}R'}\Delta_{\bB}.
\end{matrix}
\]
The composition with the unique projection $\ce_{\bA\bB}[-2]\to\yiiII\ce_{\bA\bB}[-2]$ defines visibly  an invariant operator which on locally flat structures  coincides with the nonstandard operator up to a scalar multiple. 
\end{proof}

\subsection{Proof of theorem \ref{thm} in the case of a nonvanishing torsion} In the presence of the torsion,  $\widetilde\square_{AB}S^{C'}_{C}$ is determined by a nonzero second order operator with the torsion in its leading part. Indeed, using the fact that the symmetry \eqref{torsion} of the torsion implies 
\[
\nabla_{R'B}\nabla^{R'}_{C} f=\nabla_{R'[B}\nabla^{R'}_{C]} f,
\] 
the third order terms in \eqref{obstruction} sum up to 
 \[
\begin{aligned}
 -2\nabla^{(R'}_{[A}\nabla^{C')}_{B]}\nabla_{CR'} f 
=R^{R'C'R}_{[AB] C}\nabla_{RR'}f+R^{R'C'S'}_{[AB] R'}\nabla_{CS'}f
-T^{R'C'S}_{A\ph B\ph S'}\nabla^{S'}_S\nabla_{CR'}f,
\end{aligned}
\]
and substituting the Riemannian curvatures by the analogue of the Weyl tensor according to equations  \eqref{U_def},   we  get 
\begin{equation}
\label{13}
\widetilde\square_{AB}S^{C'}_{C}=
U^{R'C'R}_{[AB]C}\nabla_{RR'}+U^{S'C'R'}_{[AB] S'}\nabla_{CR'}
+2\Pe^{(R'C')}_{[AB]}\nabla_{CR'}
-T^{R'C'S}_{A\ph B\ph S'}\nabla^{S'}_S\nabla_{CR'}.
\end{equation}
The curvature terms can be rewritten purely in terms of torsion by applying formulas in \ref{B} in appendix. In particular, 
it is visible that  the operator defined by the formula in the bottom slot is invariant if and only if the torsion vanishes. 

A way how to correct the construction in the presence of torsion is to correct the connection which defines the action of $\widetilde\square_{AB}$ on tractors. Such a correction is a one form with values in the endomorphisms of our tractor bundle $\ce_{[\al\bt]}$.
Exploring this space, we immediately find a candidate of a correction - the curvature of the standard tractor connection embedded invariantly into our space, i.e. an object defined by
\[
\Om_{a\al\gamma}^{\phantom{a\al\gamma}\bt\dt}:=
2X_{[\al|}^B X^{[\bt|}_{B'}\Om_{aB|\gamma]}^{\phantom{a}B'\phantom{\gamma}|\delta]}.
\]
In the vector notation, it acts on sections of $\ce_{[\al\bt]}$ by  \eqref{standard_curvature} as follows
\begin{equation}
\label{Om}
\Om^{A'}_A\sharp \begin{pmatrix}\sig\\ \mu^{C'}_C \\ \rho_{CD}\end{pmatrix}
=\begin{pmatrix}0 \\
T^{A'C'R}_{A\ph C\ph R'}\mu^{R'}_{R}+2U^{A'R'\ph C'}_{AC\ph R'}\sig\\ 
U_{A[CD]}^{A'R'R}\mu_{RR'}-2Q_{A[CD]R'}^{A'R'}\sig
\end{pmatrix}.
\end{equation}
Subtracting this one form from the normal tractor connection \eqref{normal_connection}, we get a new connection which behaves like the connction in the torsion free case. Indeed, such a correction
cancels second order terms in $\widetilde\square_{AB}S^{C'}_{C}$. But this is not sufficient. It turns out that terms which are of order one in the initial function and quadratic in the torsion remain. The crucial observation now is that the term in the middle slot can be written as a sum of two invariant terms and thus the whole tractor can be written as a sum of two tractors. Namely, by \eqref{dim=2} and \eqref{B1} we have
\[
U^{A'R'\ph C'}_{A\ph CR'}=-U^{A'C' R'}_{A\ph C\ph R'}+U^{A'R'C'}_{A\ph C\ph R'}=
2\Pe^{(A'C')}_{[AC]}-U^{R'A'C'}_{C\ph A\ph R'},
\]
and by equations \eqref{B3} and \eqref{B4} the second summand is given by a quadrat of the torsion and thus is invariant. Computing the corresponding splitting, we get another (finer) correction 
\begin{equation}
\label{C}
C^{A'}_A\sharp  \begin{pmatrix}\sig\\ \mu^{C'}_C \\ \rho_{CD}\end{pmatrix}=
\begin{pmatrix}
0
\\
U^{R'A'C'}_{C\ph A\ph R'}\sig
\\
\nabla_{R'[C}U^{S'(A'R')}_{D]A\ph S'}\sig-\nabla_{R'[C}U^{S'[A'R']}_{D]A\ph S'}\sig
\end{pmatrix}.
\end{equation}
It is easy to see that the connection $\widetilde{\nabla}^{A'}_A$ from equation \eqref{connection} can then be equivalently written in terms of the normal tractor connection $\nabla^{A'}_A$ and in terms of corrections \eqref{Om} and \eqref{C} as 
\[\widetilde{\nabla}^{A'}_A=\nabla^{A'}_A-\Om^{A'}_A\sharp -4C^{A'}_A\sharp.\]
We will show in the lemma below that with this connection, all terms in $\widetilde\square_{AB}S^{C'}_{C}$ cancel upon the projection to the target bundle $\ce^{A'}\yiIIi\ce_{BCD}[-1]$.

\begin{lem}
\label{istotallyskew}
If the connection defining $\widetilde\square_{AB}$ is coupled to the  modified tractor connection $\widetilde{\nabla}^{A'}_A$, then the construction described in the previous section  gives
\[
\widetilde\square_{AB}S^{C'}_{C}f \in \ce^{C'}_{[ABC]}[-1].
\]
\end{lem}
\begin{proof}
We prove an equivalent statement  $\widetilde\square_{A(B}S^{C'}_{C)}f=0$.
Symbolically
$$
\widetilde\square=\widetilde\square^{norm}-\Om\sharp\nabla-\nabla\Om\sharp+\Om\sharp \Om\sharp-4(C\sharp\nabla+\nabla C\sharp-4C\sharp C\sharp-\Om\sharp C- C\sharp\Om),
$$
where $\widetilde\square^{norm}$ is the operator from the previous section, obtained from the normal tractor connection \eqref{normal_connection}. Evidently, the actions of $\nabla C\sharp, C\sharp C\sharp, \Om\sharp C$ and $ C\sharp\Om$ on the splitting \eqref{split} vanish while
for the other terms using \eqref{Om} and \eqref{C} we compute
\[
\begin{aligned}
\Om_{R'A}\sharp\nabla^{R'}_B S^{C'}_C &=
-T^{R'C'S}_{A\ph C\ph S'}\nabla_{BR'}\nabla^{S'}_S +2\Pe^{(R'C')}_{[AC]}\nabla_{BR'}-U^{S'R'C'}_{C\ph A\ph S'}\nabla_{BR'}
\\
\nabla_{R'A}\Om^{R'}_B\sharp S^{C'}_C &=
\nabla_{R'A}T^{R'C'S}_{B\ph C\ph S'}\nabla^{S'}_S -2U^{C'R'R}_{B[AC]}\nabla_{RR'}
\\
\Om_{R'A}\sharp \Om^{R'}_B\sharp S^{C'}_C &=
T_{A\ph C\ph U'}^{R'C'U}T^{U'\phantom{RB}S}_{UR'BS'}\nabla^{S'}_S 
\\
C_{R'A}\sharp\nabla^{R'}_B S^{C'}_C &= \tfrac12 U^{S'R'C'}_{C\ph A\ph S'}\nabla_{BR'}
\end{aligned}
\]
Summarizing the results above, we get
\[
\begin{aligned}
\widetilde\square_{AB}S^{C'}_{C}=
\widetilde\square^{norm}_{AB}S^{C'}_{C}+
2T^{R'C'S}_{[A|CS'}\nabla_{|B]R'}\nabla^{S'}_S 
-(\nabla_{R'[A}T^{R'C'S}_{B]CS'})\nabla^{S'}_{S}
\\
+T_{[A| C\ph U'}^{R'C'U}T^{U'\phantom{RB}S}_{UR'|B]S'}\nabla^{S'}_S 
-\Pe^{(R'C')}_{[AC]}\nabla_{BR'}+\Pe^{(R'C')}_{[BC]}\nabla_{AR'}
\\
+U^{C'R'R}_{B[AC]}\nabla_{RR'}-U^{C'R'R}_{A[BC]}\nabla_{RR'}
-U^{S'R'C'}_{C[A|S'}\nabla_{|B]R'}
\end{aligned}
\]
Now we apply the formula   \eqref{obstruction}  for $\square^{norm}_{AB}S^{C'}_{C}f$, we symmetrize in $B$ and $C$, and we use the symmetry \eqref{torsion} of the torsion. Then we get
\[
\begin{aligned}
\widetilde\square_{A(B}S^{C'}_{C)}=&U^{(R'C')R}_{A(BC)}\nabla_{RR'}+U^{(S'C')R'}_{A(B|S'|}\nabla_{C)R'}
+\Pe^{(R'C')}_{[AB]}\nabla_{CR'}+\Pe^{(R'C')}_{[AC]}\nabla_{BR'}
\\&
-T^{R'C'S}_{A(B| S'}\nabla^{S'}_{S|}\nabla_{C)R'}
+T^{R'C'S}_{A(B|S'}\nabla_{|C)R'}\nabla^{S'}_S 
-\tfrac12(\nabla_{R'(B}T^{R'C'S}_{C)AS'})\nabla^{S'}_{S}
\\&
+\tfrac12T_{A(B|U'}^{R'C'U}T^{U'\phantom{RB}S}_{UR'|C)S'}\nabla^{S'}_S 
-\tfrac12\Pe^{(R'C')}_{[AC]}\nabla_{BR'}-\tfrac12\Pe^{(R'C')}_{[AB]}\nabla_{CR'}
\\&
+\tfrac12U^{C'R'R}_{B[AC]}\nabla_{RR'}+\tfrac12U^{C'R'R}_{C[AB]}\nabla_{RR'}
\\&
-\tfrac12U^{S'R'C'}_{(B|AS'|}\nabla_{C)R'}+\tfrac12U^{S'R'C'}_{(BC)S'}\nabla_{AR'}
\end{aligned}
\]
The second order terms are identical up to a commutation of the derivatives, and thus they sum up to a first order term quadratic in the torsion. We add this term to the other term quadratic in the torsion. Next, we collect and rewrite terms involving the Rho tensor,  terms involving the curvature $U_{ab}{}^{C'}_{D'}$ and terms involving $U_{ab}{}^C_D$. For rewriting terms involving (trace of) $U_{ab}{}^{C'}_{D'}$, we use  consequence \eqref{B2} of the Bianchi identity given below. Otherwise, we use a manipulation with indices and algebraic operations only. We get an equivalent equation
\[
\begin{aligned}
2\cdot\widetilde\square_{A(B}S^{C'}_{C)}=&
T_{A(B|U'}^{S'C'U}T^{U'R'R}_{U|C)S'}\nabla_{R'R}
-(\nabla^{S'}_{(B}T^{R'C'R}_{C)AS'})\nabla_{R'R}
\\& 
+\Pe^{(C'R')}_{[AB]}\nabla_{CR'}f+\Pe^{(C'R')}_{[AC]}\nabla_{BR'}
\\&
+U^{S'[C'R']}_{A(B|S'}\nabla_{|C)R'}f-U^{S'(C'R')}_{A(B|S'}\nabla_{|C)R'}
+U^{S'(C'R')}_{BCS'}\nabla_{AR'}
\\&
+\left(U^{(C'R')R}_{A(BC)}+U^{[C'R']R}_{B[AC]}
+U^{[C'R']R}_{C[AB]}\right)\nabla_{RR'}.
\end{aligned}
\]
To prove that this is zero, we write curvature terms in the last line of the previous equation in terms of  torsion.  Namely, by  \eqref{B5} we have
\[
U^{[C'R']R}_{B\ph [A\ph C]}+U^{[C'R']R}_{C\ph [A\ph B]}=
-T^{S'[C'|F}_{A\ph (B|\ph F'}T^{F'|R']R}_{F\ph |C)\ph S'}
- U^{S'[C'R']}_{A(B\ph |S'|}\dt^R_{C)},
\] 
and by  \eqref{B6} and \eqref{B1} we have
\[
\begin{aligned}
U^{(C'R')R}_{A\ph (B\ph C)}=&
\nabla^{S'}_{(B}T^{R'C'R}_{C)A\ph S'}
-T^{S'(C'|F}_{A\ph (B\ph |F'}T^{F'|R')R}_{F\ph |C\ph )S'}
-\Pe^{(C'R')}_{[AB]}\dt^R_C-\Pe^{(C'R')}_{[AC]}\dt^R_B
\\ &
+U^{S'(C'R')}_{A(B|S'|}\dt^R_{C)}
-U^{S'(C'R')}_{B\ph C\ph S'}\dt^R_{A},
\end{aligned}
\]
The substitution then gives $\widetilde\square_{A(B}S^{C'}_{C)}f=0$ and this proves the lemma.
\end{proof}
Having proven this technically difficult lemma it is now easy to finish the proof of theorem \ref{thm} in the presence of a nonzero torsion. Namely, the lemma shows that $\widetilde\square_{\bA}S^{A'}_{B}$ has values in  bundle $\ce^{A'}\otimes\ce_{[\bA B]}[-1]$ only. On the other hand, the target bundle of the operator $\square_{\bA\bB}$ does not appear in the decomposition of the tensor product of bundles $\ce^{A'}\otimes\ce_{[\bA B]}[-1]$ and $\ce^{C'}_C$. Hence  $\square_{\bA\bB}$ transforms exclusively by terms in  $\ce^{A'}\otimes\yiIIi\ce_{\bA B}[-1]$ nad thus is invariant. A straightforward application of formulas \eqref{Laplace} and \eqref{connection} then gives the formula \eqref{formula}.

Let us remark at the end that the invariance of \eqref{formula} can be also checked by a direct but lengthy computation of its transformation under a change of the Weyl connection by a repetetive application of formulas summarized in Remark \ref{rem}.

\section{Bianchi identity} 
Bianchi identity gives the relation between the torsion and all components of the curvature except its harmonic part.
In terms of abstract (tensor) indices it reads as
\[
R_{[abc]}^{\phantom{[ab}e}=
-\nabla_{[a}T_{bc]}^{\phantom{bc}e}+T_{[ab}^{\phantom{ab}f}T_{|f|d]}^{\phantom{|f|d}e}.
\]
The precise relations for individual components of the curvature are obtained by taking suitable algebraic operations and by taking into accont the symmetries of the torsion and the curvature. For the sake of simplicity we express  the components of the curvature tensor $U_{abc}^{\phantom{ab}e}$, defined by \eqref{U_def},  rather  then the components of the Riemannian curvature.
\begin{lem}
\label{B}
Components of the curvature tensors $U^{A'B'C}_{A\ph B\ph D}$ and $U^{A'B'C'}_{A\ph B\ph D'}$ are related to the torsion as follows. Trace terms:
\begin{align}
\label{B1}
U_{abE}^{\phantom{ab}E}&=-U_{abE'}^{\phantom{ab}E'}=
-\tfrac{1}{q+2}\nabla_{e}T_{ab}^{\phantom{ab}e}=2\Pe_{[ab]}.
\\
\label{B2}
U^{(A'B')E}_{E\ph [A\ph B]}&=U^{E'(A'B')}_{[A\ph B]\ph E'}=
U^{[A'B']E}_{E\ph (A\ph B)}=U^{E'[A'B']}_{(A\ph B)\ph E'}=0.
\\
U^{(A'B')E}_{E\ph (A\ph B)}&=U^{E'(A'B')}_{(A\ph B)\ph E'}=-\tfrac{1}{q}T^{E'(A'|F}_{E\ph (A|\ph F'}T^{F'|B')E}_{F\ph |B)\ph E'}, \label{B3}
\\
U^{[A'B']E}_{E\ph [A\ph B]}&=U^{E'[A'B']}_{[A\ph B]\ph E'}=\tfrac{1}{q+4}T^{E'[A'|F}_{E\ph [A|\ph F'}T^{F'|B']E}_{F\ph |B]\ph E'}, \label{B4}
\end{align}
Tracefree terms:
\begin{align}
\label{B5}
U^{[A'B']E}_{A\ph [B\ph C]}&=
\tfrac13(T^{E'[A'|F}_{A\ph [B|\ph F'}T^{F'|B']E}_{F\ph |C]\ph E'}
-T^{E'[A'|F}_{B\ph C\ph F'}T^{F'|B']E}_{F\ph A\ph E'}
\\ &
+U^{E'[A'B']}_{A[B\ph |E'|}\dt^E_{C]}
-U^{E'[A'B']}_{B\ph C\ph E'}\dt^E_{A}),\nonumber
\\
\label{B6}
U^{(A'B')E}_{A\ph (B\ph C)}&=
\nabla^{E'}_{(B}T^{A'B'E}_{C)A\ph E'}
-T^{E'(A'|F}_{A\ph (B\ph |F'}T^{F'|B')E}_{F\ph |C\ph )E'}
+U^{A'B'E'}_{A(B|E'|}\dt^E_{C)}
\\ &
+U^{E'(A'B')}_{A(B|E'|}\dt^E_{C)}
-U^{E'(A'B')}_{B\ph C\ph E'}\dt^E_{A},\nonumber
\\
U^{A'B'E}_{[A\ph B\ph C]}&=\tfrac14(-\nabla^{E'}_{[A}T^{A'B'E}_{B\ph C] E'}
+2T^{E'(A'|F}_{[A\ph B\ph |F'}T^{F'|B')E}_{F\ph |C]\ph E'}
+U^{A'B'E'}_{[AB|E'|}\dt^E_{C]}), \label{B7}
\\
U_{\ph A \ph B\ph  E'}^{(A'B'C')}&=\tfrac{1}{q-2}(\nabla^{(A'}_{E}T^{B'C')E}_{A\ph B\ph E'}
-2T^{(A'B'|F}_{E\ph [A\ph |F'}T^{F'|C')E}_{F\ph|B\ph]E'}
+U^{(A'B'|I'|}_{\ph A\ph B\ph I'}\dt^{C')}_{E'}), \label{B8}
\\
U_{(A B) E'}^{A'B'C'}&=\tfrac{2}{1-q}(T^{C'[A'|F}_{E\ph(A\ph |F'}T^{F'|B']E}_{F\ph |B)E'}
-U^{C'[A'|E|}_{E\ph (A\ph B)}\dt^{B']}_{E'}), \label{B9}
\end{align}
\end{lem}
\begin{proof}
First observe that  $\partial\Pe_{[abc]}^{\phantom{abc}e}=0$ and thus the Riemannian curvature tensor  on the left hand side of the Bianchi identity can be equivalently replaced by the curvature tensor $U$. Then expanding the alternations, the Bianchi identity in detail has the form as follows
\[
\begin{array}{l}
U_{A\ph B\ph C}^{A'B'E}\dt^{C'}_{E'}-U_{A\ph B\ph E'}^{A'B'C'}\dt^{E}_{C}
+U_{C\ph A\ph B}^{C'A'E}\dt^{B'}_{E'}
-U_{C\ph A\ph E'}^{C'A'B'}\dt^{E}_{B}
+U_{B\ph C\ph A}^{B'C'E}\dt^{A'}_{E'}
\\
-U_{B\ph C\ph E'}^{B'C'A'}\dt^{E}_{A}
=-\nabla^{A'}_{A}T^{B'C'E}_{B\ph C\ph E'}-\nabla^{C'}_{C}T^{A'B'E}_{A\ph B\ph E'}
-\nabla^{B'}_{B}T^{C'A'E}_{C\ph A\ph E'}
\\
+T^{A'B'F}_{A\ph B\ph F'}T^{F'C'E}_{F\ph C\ph E'}
+T^{C'A'F}_{C\ph A\ph F'}T^{F'B'E}_{F\ph B\ph E'}
+T^{B'C'F}_{B\ph C\ph F'}T^{F'A'E}_{F\ph A\ph E'}.
\end{array}
\]
Taking  the contractions $\dt^{E'}_{C'}$ and $\dt^C_E$ we get 
\begin{equation}
\label{Bianchi_1}
\begin{array}{l}
2U^{A'B'E}_{A\ph B\ph C}-U^{A'B'E'}_{A\ph B\ph E'}\dt^E_C
+U_{C\ph A\ph B}^{B'A'E}
-U^{E'A'B'}_{C\ph A\ph E'}\dt^E_B
+U^{B'A'E}_{B\ph C\ph A}-U^{B'E'A'}_{B\ph C\ph E'}\dt^E_A
\\
=
-\nabla^{E'}_C T^{A'B'E}_{A\ph B\ph E'}
+T^{E'A'F}_{C\ph A\ph F'}T^{F'B'E}_{F\ph B\ph E'}
+T^{B'E'F}_{B\ph C\ph F'}T^{F'A'E}_{F\ph A\ph E'},
\end{array}
\end{equation}
respectively 
\begin{equation}
\label{Bianchi_2}
\begin{array}{l}
U_{A\ph B\ph E}^{A'B'E}\dt^{C'}_{E'}-qU_{A\ph B\ph E'}^{A'B'C'}
+U_{E\ph A\ph B}^{C'A'E}\dt^{B'}_{E'}
-U_{B\ph A\ph E'}^{C'A'B'}
+U_{B\ph E\ph A}^{B'C'E}\dt^{A'}_{E'}
-U_{B\ph A\ph E'}^{B'C'A'}
\\
=-\nabla^{C'}_{E}T^{A'B'E}_{A\ph B\ph E'}
+T^{C'A'F}_{E\ph A\ph F'}T^{F'B'E}_{F\ph B\ph E'}
+T^{B'C'F}_{B\ph E\ph F'}T^{F'A'E}_{F\ph A\ph E'}.
\end{array}
\end{equation}
Equation \eqref{B1} is obtained by  taking both contractions simultaneously and by applying the definition \eqref{U_def} and the coclosedness condition \eqref{coclosedness} of the curvature tensor $U$  and  the trace-freeness and the symmetries \eqref{torsion} of the torsion $T$. The other trace-parts \eqref{B2}, \eqref{B3} and \eqref{B4}  are obtained by taking suitable algebraic projections in the contraction $\dt^B_E$ of  \eqref{Bianchi_1} which by \eqref{B1} simplifies to
\[
2U^{A'B'E}_{A\ph E\ph C}
-qU^{E'A'B'}_{C\ph A\ph E'}+U^{B'A'E}_{E\ph C\ph A}
-U^{B'E'A'}_{A\ph C\ph E'}=T^{B'E'F}_{E\ph C\ph F'}T^{F'A'E}_{F\ph A\ph E'}.
\]
Equations \eqref{B5}, \eqref{B6} and \eqref{B7} then follows from \eqref{Bianchi_1} by taking alternations $[A'B']$ and $[BC]$, symmetrizations $(A'B')$ and $(BC)$ and  alternation  $[ABC]$ respectively, and by applying the formulas for the traces of $U$. Similarly, equations \eqref{B8} and \eqref{B9} follow from \eqref{Bianchi_2} by taking symmetrizations $(A'B'C')$ and  $(AB)$ respectively. 
\end{proof}


\begin{thebibliography}{99}
\bibitem{Eastwood} Eastwood M., Slov\' ak J., {\it Semi-holonomic Verma modules}, J. of Algebra, 197 (1997), 424-448
\bibitem{Gover} Gover A.R., Slov\' ak J., {\it Invariant local twistor calculus for quaternionic structures and related geometries,} J. Geom. Phys. 32, No.1 (1999) 14-56
\bibitem{Navrat} N\'avrat A., {\it Nonstandard invariant operators on
quaternionic geometries}, MSc Thesis, Masaryk University in Brno, 2004
\bibitem{parabook} A. \v{C}ap and J. Slov\'ak, 
Parabolic Geometries I: Background and General Theory,
Math. Surv. and Monographs, {\bf 154},
Amer. Math. Soc., Providence, RI, 2009.
\bibitem{prolongation} M. Hammerl, P. Somberg, V. Sou\v cek, J. \v Silhan, 
Invariant prolongation of overdetermined PDE's in projective, conformal and Grassmannian geometry,
Annals of Global Analysis and Geometry, Springer, 2012, 0232-704X.\
\bibitem{BoeCollingwood} Boe, Brian D.; Collingwood, David H.  Multiplicity free categories of highest weight representations. I, II. Commun. Algebra 18, No.4, 947-1032, 1033-1070 (1990)
\bibitem{Boe}B. Boe. Homomorphisms between generalized Verma modules. {\it Trans. Amer. Math. Soc.} 356 (1): 159-184, (2004).
\bibitem{Graham} C.R. Graham, Conformally invariant powers of
  the Laplacian. II. Nonexistence. J. London Math. Soc. \textbf{46}
  no. 3 (1992) 566--576.
  \bibitem{translation} A. \v Cap, Translation of natural operators on manifolds
with AHS--structures, Archivum Math. (Brno) 32, 4 (1996), 249--266,
lectronically available at www.emis.de
  \bibitem{Qcurv} Gover A.R., Conformal de Rham Hodge theory and operators generalising the Q-curvature, http://arxiv.org/abs/math/0404004
\end{thebibliography}
\end{document}